\documentclass[11pt,letterpaper,twoside,reqno]{amsart}
\usepackage{amssymb,mathbbol,bm}
\usepackage[text={6.5in,9in},centering,letterpaper]{geometry}

\newcommand{\e}{\ensuremath{{\rm e}}}
\newcommand{\mat}[1]{\ensuremath{\bm{#1}}}
\newcommand{\E}{\ensuremath{\mathbb{E}}}
\newcommand{\Prob}[1]{\ensuremath{\mathbb{P}\left(#1\right)}}
\newcommand{\R}{\ensuremath{\mathbb{R}}}
\newcommand{\supoverinfball}{\ensuremath{\max_{\infnorm{\mat{u}} = 1}}}
\newcommand{\supovertball}{\ensuremath{\max_{\norm{\mat{u}}_2 = 1}}}
\newcommand{\supoverqball}{\ensuremath{\max_{\qnorm{\mat{u}} = 1}}}

\newcommand{\norm}[1]{\ensuremath{\left\|#1\right\|}}
\newcommand{\pnorm}[1]{\ensuremath{\left\|#1\right\|_p}}
\newcommand{\qnorm}[1]{\ensuremath{\left\|#1\right\|_q}}
\newcommand{\infnorm}[1]{\ensuremath{\left\|#1\right\|_\infty}}
\newcommand{\infonorm}[1]{\ensuremath{\left\|#1\right\|_{\infty\rightarrow 1}}}
\newcommand{\inftnorm}[1]{\ensuremath{\left\|#1\right\|_{\infty\rightarrow 2}}}

\newcommand{\twoinfnorm}[1]{\ensuremath{\left\|#1\right\|_{2 \rightarrow \infty}}}
\newcommand{\infone}{\ensuremath{\infty\!\rightarrow\!\!1}}
\newcommand{\inftwo}{\ensuremath{\infty\!\rightarrow\!\!2}}
\newcommand{\infp}{\ensuremath{\infty\!\rightarrow\!p}}
\newcommand{\twoinf}{\ensuremath{2\!\!\rightarrow\!\!\infty}}
\newcommand{\colnorm}[1]{\ensuremath{\left\|#1\right\|_{\text{\rm col}}}}
\newcommand{\cutnorm}[1]{\ensuremath{\left\|#1\right\|_{\text{\rm C}}}}
\newcommand{\frobnorm}[1]{\ensuremath{\left\|#1\right\|_{\text{\rm F}}}}
\newcommand{\asymO}[1]{\ensuremath{\mathop{\text{\rm O}}\!\left(#1\right)}}

\newcommand{\var}[1]{\ensuremath{\mathrm{Var}(#1)}}

\DeclareMathOperator{\sgn}{sgn}
\DeclareMathOperator{\trace}{trace}
\newcommand{\eqqcolon}{=\mathrel{\mathop:}} 

\newtheorem{thm}{Theorem}
\newtheorem{cor}{Corollary}

\newtheorem{prop}{Proposition}

\begin{document}
\title[Error bounds for random matrix approximation]{Error bounds for random matrix approximation schemes}

\author{A.~Gittens \and J.~A.~Tropp}

\begin{abstract}
Randomized matrix sparsification has proven to be a fruitful technique for producing faster algorithms in applications ranging from graph partitioning to semidefinite programming. 
In the decade or so of research into this technique, the focus has been---with few exceptions---on ensuring the quality of approximation in the spectral and Frobenius norms. For certain graph algorithms, however, the $\infone$ norm may be a more natural measure of performance. 

This paper addresses the problem of approximating a real matrix $\mat{A}$ by a sparse random matrix $\mat{X}$ with respect to several norms. It provides the first results on approximation error in the $\infone$ and $\inftwo$ norms, and it uses a result of Lata\l{}a to study approximation error in the spectral norm. These bounds hold for a reasonable family of random sparsification schemes, those which ensure that the entries of $\mat{X}$ are independent and average to the corresponding entries of $\mat{A}$. Optimality of the $\infone$ and $\inftwo$ error estimates is established. Concentration results for the three norms hold when the entries of $\mat{X}$ are uniformly bounded. The spectral error bound is used to predict the performance of several sparsification and quantization schemes that have appeared in the literature; the results are competitive with the performance guarantees given by earlier scheme-specific analyses.
\end{abstract}
\maketitle

\section{Introduction}

Massive datasets are ubiquitous in modern data processing. Classical dense matrix algorithms are poorly suited to such problems because their running times scale superlinearly with the size of the matrix. When the dataset is sparse, one prefers to use sparse matrix algorithms, whose running times depend more on the sparsity of the matrix than on the size of the matrix. Of course, in many applications the matrix is \emph{not} sparse. Accordingly, one may wonder whether it is possible to approximate a computation on a large dense matrix with a related computation on a sparse approximant to the matrix.

Let $\|\cdot\|$ be a norm on matrices. We may phrase the following question: Given a matrix $\mat{A}$, how can one efficiently generate a sparse matrix $\mat{X}$ for which the approximation error $\|\mat{A} - \mat{X}\|$ is small? 

In the seminal papers \cite{AM07,AM01}, Achlioptas and McSherry demonstrate that one can bound, \emph{a priori}, the spectral and Frobenius norm errors incurred when using a particular random approximation scheme. They use this scheme as the basis of an efficient algorithm for calculating near optimal low-rank approximations to large matrices. In a related work \cite{AHK06}, Arora, Hazan, and Kale present another randomized scheme for computing sparse approximants with controlled Frobenius and spectral norm errors.

The literature has concentrated on the behavior of the approximation error in the spectral and Frobenius norms; however, these norms are not always the most natural choice. For instance, the problem of graph sparsification is naturally posed as a question of preserving the cut-norm of a graph Laplacian. The strong equivalency of the cut-norm and the $\infone$ norm suggests that, for graph-theoretic applications, it may be fruitful to consider the behavior of the $\infone$ norm under sparsification. In other applications, e.g., the column subset selection algorithm in \cite{Tropp09}, the $\inftwo$ norm is the norm of interest.

This paper investigates the errors incurred by approximating a fixed real matrix with a random matrix. Our results apply to any scheme in which the entries of the approximating matrix are independent and average to the corresponding entries of the fixed matrix. Our main contribution is a bound on the expected $\infp$ norm error, which we specialize to the case of the $\infone$ and $\inftwo$ norms. We also use a result of Lata{\l}a \cite{Lat04} to find an optimal bound on the expected spectral approximation error, and we establish the subgaussianity of the spectral approximation error.

\subsection{Summary of results}
Consider the matrix $\mat{A}$ as an operator from $\ell^n_\infty$ to $\ell^m_p$, where $1 \leq p \leq \infty$. The operator norm of $\mat{A}$ is defined as 
\[
\|\mat{A}\|_{\infty \rightarrow p} = \max_{\mat{u} \neq \mathbf{0}} \frac{\|\mat{A}\mat{u}\|_p}{\|\mat{u}\|_\infty}
.\]

The major novel contributions of this paper are estimates of the $\infone$ and $\inftwo$ norm approximation errors induced by random matrix sparsification schemes. Let $\mat{A}$ be a target matrix. Suppose that $\mat{X}$ is a random matrix with independent entries that satisfies $\E \mat{X} = \mat{A}.$ Then
\[
\E\norm{\mat{A} - \mat{X}}_{\infty \rightarrow 1} \leq 2\left[\sum\nolimits_k\left(\sum\nolimits_j \var{X_{jk}} \right)^{1/2} + \sum\nolimits_j\left(\sum\nolimits_k \var{X_{jk}}\right)^{1/2} \right]
\] 
and
\[
\E\norm{\mat{A} - \mat{X}}_{\infty \rightarrow 2} \leq 2\left(\sum\nolimits_{jk} \var{X_{jk}} \right)^{1/2} + 2\sqrt{m}\min_{\mat{D}}\max_j\left(\sum\nolimits_k \frac{\var{X_{jk}}}{d_k^2}\right)^{1/2}
\]
where $\mat{D}$ is a positive diagonal matrix with $\trace (\mat{D}^2) = 1.$ We complement these estimates with tail bounds for the approximation error in each norm.

We also note that 
\[
\E\norm{\mat{A} - \mat{X}} \leq {\rm C}\left[\max_j\left(\sum\nolimits_k \var{X_{jk}}\right)^{1/2} + \max_k\left(\sum\nolimits_j \var{X_{jk}}\right)^{1/2} + \left(\sum\nolimits_{jk} \E(X_{jk} - a_{jk})^4\right)^{1/4} \right]
\]
where ${\rm C}$ is a universal constant. This estimate follows from Lata{\l}a's work \cite{Lat04}. We use our spectral norm results to analyze the sparsification and quantization schemes in \cite{AHK06} and \cite{AM07}, and we show that our analysis yields error estimates competitive with those established in the respective works.

While on the path to proving the $\infone$ and $\inftwo$ norm approximation error estimates, we derive analogous relations of independent interest that bound the expected norm of a random matrix $\mat{Z}$ with independent, zero-mean entries:
\[
\E\norm{\mat{Z}}_{\infty \rightarrow 1} \leq 2\E\left(\colnorm{\mat{Z}} + \colnorm{\mat{Z}^T}\right),
\]
where $\colnorm{\mat{A}}$ is the sum of the $\ell_2$ norms of the columns of $\mat{A}$, and
\[
\E\norm{\mat{Z}}_{\infty \rightarrow 2} \leq 2\E\frobnorm{\mat{Z}} + 2\min_{\mat{D}}\E\norm{\mat{Z}\mat{D}^{-1}}_{2\rightarrow \infty},
\]
where $\mat{D}$ is a diagonal positive matrix with $\trace(\mat{D}^2) = 1$. More generally,
\[
\E\norm{\mat{Z}}_{\infty\rightarrow p} \leq 2\E\norm{\sum\nolimits_k \varepsilon_k\mat{z}_k}_p + 2\supoverqball \E\sum\nolimits_k\left|\sum\nolimits_j \varepsilon_j Z_{jk} u_j\right|,
\]
where $q$ is the conjugate exponent to $p$ and $\mat{z}_k$ is the $k$th column of $\mat{Z}$. Here, $\{\varepsilon_k\}$ is a sequence of independent random variables uniform on $\{\pm 1\}.$

\subsection{Outline}

Section \ref{sec:literature} offers an overview of the related strands of research, and Section \ref{sec:background} introduces the reader to our notations. In Section \ref{sec:inftopbound}, we establish the foundations for our $\infone$ and $\inftwo$ error estimates: an estimate of the expected $\infp$ norm of a random matrix with independent, zero-mean entries and a tail bound for this quantity. In Sections \ref{sec:inf1norm} and \ref{sec:inf2norm}, these estimates are specialized to the cases $p=1$ and $p=2$, respectively, and we establish the optimality of the resulting expressions. The bounds are provided in both their most generic forms and ones more suitable for applications. In Section \ref{sec:spectralerrbound}, we use a result due to Lata\l{}a \cite{Lat04} to find an optimal estimate of the spectral approximation error. A deviation bound is provided for the spectral norm which captures the correct (subgaussian) tail behavior. We show that our estimates applied to the sparsification schemes in \cite{AHK06} and \cite{AM07} recover performance guarantees comparable with those obtained from the original analyses, under slightly weaker hypotheses.

\section{Related Work}
\label{sec:literature}

In recent years, much attention has been paid to the problem of using sampling methods to approximate linear algebra computations efficiently. Such algorithms find applications in areas like data mining, computational biology, and other areas where the relevant matrices may be too large
to fit in RAM, or large enough that the computational requirements of standard algorithms become prohibitive. Here we review the streams of literature that have motivated and influenced this work.

\subsection{Randomized sparsification}
The seminal research of Frieze, Kannan, and Vempala \cite{FKV98,FKV04} on using random sampling to decrease the running time of linear algebra algorithms focuses on approximations to $\mat{A}$ constructed from random subsets of its rows. The subsequent influential works of Drineas, Kannan, and Mahoney \cite{DKM06a,DKM06b, DKM06c} also analyze the performance of Monte Carlo algorithms which sample from the columns or rows of $\mat{A}$. 

In \cite{AM01}, Achlioptas and McSherry advance a different approach: instead of using low-rank approximants, they sample the entries of $\mat{A}$ to produce a sparse matrix $\mat{X}$ that has a spectral decomposition close to that of $\mat{A}$. This transition from row/column sampling to independent sampling of the entries allows them to bring to bear powerful techniques from random matrix theory. In \cite{AM01}, their main tool is a result on the concentration of the spectral norm of a random matrix with independent, zero-mean, bounded entries. In a follow-up paper, \cite{AM07}, they obtain better estimates by using a sharper concentration result derived from Talagrand's inequality. They point out the importance of using schemes which sparsify a given entry in $\mat{A}$ with a probability proportional to its magnitude. Such adaptive sparsification schemes keep the variance of the individual entries small, which tends to keep the error in approximating $\mat{A}$ with $\mat{X}$ small. Their scheme requires either two passes through the matrix or prior knowledge of an upper bound for the largest magnitude in $\mat{A}$.

In \cite{AHK06}, Arora, Hazan, and Kale describe a random sparsification algorithm which partially quantizes its inputs and requires only one pass through the matrix. They use an epsilon-net argument and Chernoff bounds to establish that with high probability the resulting approximant has small error and high sparsity. 

\subsection{Probability in Banach spaces}
In \cite{RV07}, Rudelson and Vershynin take a different approach to the Monte Carlo methodology for low-rank approximation. They consider $\mat{A}$ as a linear operator between finite-dimensional Banach spaces and apply techniques of probability in Banach spaces: decoupling, symmetrization, Slepian's lemma for Rademacher random variables, and a law of large numbers for operator-valued random variables. They show that, if $\mat{A}$ can be approximated by any rank-$r$ matrix, then it is possible to obtain an accurate rank-$k$ approximation to $\mat{A}$ by sampling $\asymO{r\log r}$ rows of $\mat{A}$. Additionally, they quantify the behavior of the $\infone$ and $\inftwo$ norms of random submatrices.

Our methods are similar to those of Rudelson and Vershynin in \cite{RV07} in that we consider $\mat{A}$ as a linear operator between finite-dimensional Banach spaces and use some of the same tools of probability in Banach spaces. Whereas Rudelson and Vershynin consider the behavior of the norms of random submatrices of $\mat{A}$, we consider the behavior of the norms of matrices formed by randomly sparsifying (or quantizing) the entries of $\mat{A}$. This yields error bounds applicable to schemes that sparsify or quantize matrices entrywise. Since some graph algorithms depend more on the number of edges in the graph than the number of vertices, such schemes may be useful in developing algorithms for handling large graphs.

\subsection{Random sampling of graphs}
The bulk of the literature has focused on the behavior of the spectral and Frobenius norms under randomized sparsification, but the $\infone$ norm occurs naturally in connection with graph theory. Let us review the relevant ideas.

Consider a weighted simple graph $G=(V,E,\omega)$ with adjacency matrix
$\mat{A}$ given by 
\[
a_{jk}=\begin{cases}
\omega_{jk}, & (j,k)\in E\\
0, & \text{ otherwise}.\end{cases}
\]
A \emph{cut} is a partition of the vertices into two blocks: $V=S\cup\overline{S}$. The \emph{cost} of a cut is the sum of the weights of all edges in $E$ which have one vertex in $S$ and one vertex in
$\overline{S}$.
Several problems relating to cuts are of considerable practical interest. In particular, the \textsc{maxcut} problem, to determine the cut of maximum cost in a graph, is common in computer science applications. The cuts of maximum cost are exactly those which realize the \emph{cut-norm} of the adjacency matrix, which is defined as 
\[
\cutnorm{\mat{A}}=\max_{S\subset E}\left|\sum\nolimits_{(j,k)\in E}\omega_{jk}(\mathbb{1}_{S})_{j}(\mathbb{1}_{\overline{S}})_{k}\right|,
\]
where $\mathbb{1}_{S}$ is the indicator vector for $S$. Finding the cut-norm of a general matrix is {\sf NP}-hard, but in \cite{AN04}, the authors offer a randomized polynomial-time algorithm which finds a submatrix $\tilde{\mat{A}}$ of $\mat{A}$ such that $|\sum_{jk}\tilde{a}_{jk}|\geq0.56\cutnorm{\mat{A}}$. One crucial point in the derivation of the algorithm is the fact that
the $\infone$ norm is strongly equivalent with the cut-norm: 
\[
\cutnorm{\mat{A}}\leq\infonorm{\mat{A}}\leq4\cutnorm{\mat{A}}.
\]

In his thesis \cite{Kar95th} and the sequence of papers \cite{Kar94a,Kar94b,Kar96}, Karger introduces the idea of random sampling to increase the efficiency of calculations with graphs, with a focus on cuts. In \cite{Kar96}, he shows that by picking each edge of the graph with a probability inversely proportional to the density of edges in a neighborhood of that edge, one can construct a \emph{sparsifier}, i.e., a graph with the same vertex set and significantly fewer edges that preserves the value of each cut to within a factor of $(1\pm\epsilon)$.

In \cite{SS08}, Spielman and Srivastava improve upon this sampling scheme, instead keeping an edge with probability proportional to its \emph{effective resistance}---a measure of how likely it is to appear in a random spanning tree of the graph. They provide an algorithm which produces a sparsifier with $\asymO{(n\log n)/\epsilon^{2}}$ edges, where $n$ is the number of vertices in the graph. They obtain this result by reducing the problem to the behavior of projection matrices $\Pi_{G}$ and $\Pi_{G^{\prime}}$ associated with the original graph and the sparsifier, and appealing to a spectral norm concentration result.

The $\log n$ factor in \cite{SS08} seems to be an unavoidable consequence of using spectral norm concentration. In \cite{BSS08}, Batson et. al. prove that the $\log n$ factor is not intrinsic: they establish that every graph has a sparsifier that has $\asymO{n}$ edges. The proof is constructive and provides a deterministic algorithm for constructing such optimal sparsifiers in $\asymO{n^3m}$ time, where $m$ is the number of edges in the original graph.

The algorithm of \cite{BSS08} is clearly not suitable for sparsifying graphs with a large number of vertices. Part of our motivation for investigating the $\infone$ approximation error is the belief that the equivalence of the cut-norm with the $\infone$ norm means that matrix sparsification in the $\infone$ norm might be useful for efficiently constructing optimal sparsifiers for such graphs.

\section{Background}
\label{sec:background} 

We establish the notation used in the sequel and introduce our key technical tools.

All quantities are real. The $k$th column of the matrix $\mat{A}$ is denoted by $\mat{a}_{k}$ and the entries are denoted $a_{jk}.$

For $1\leq p\leq\infty$, the $\ell_p$ norm of $\mat{x}$ is written as $\pnorm{\mat{x}}$. We treat $\mat{A}$ as an operator from $\ell_{p}^{n}$ to $\ell_{q}^{m}$, and the $p\rightarrow q$ operator norm of $\mat{A}$ is written as $\norm{\mat{A}}_{p\rightarrow q}$. The $p\rightarrow q$ and $q^{\prime} \rightarrow p^{\prime}$ norms are dual in the sense that 
\[
\norm{\mat{A}}_{p\rightarrow q}=\norm{\mat{A}^{T}}_{q^{\prime}\rightarrow p^{\prime}}.
\]

This paper is concerned primarily with the spectral norm and the $\infone$ and $\inftwo$ norms. The spectral norm $\norm{\mat{A}}$ is the largest singular value of $\mat{A}$. The $\infone$ and $\inftwo$ norms do not have such nice interpretations, and they are {\sf NP}-hard to compute for general matrices \cite{Rohn00}. We remark that $\infonorm{\mat{A}}=\norm{\mat{A}\mat{x}}_{1}$ and $\inftnorm{\mat{A}}=\norm{\mat{A}\mat{y}}_{2}$ for certain vectors $\mat{x}$ and $\mat{y}$ whose components take values $\pm1$.

An additional operator norm, the $\twoinf$ norm, is also of interest: it is the largest $\ell_{2}$ norm achieved by a row of $\mat{A}$. We encounter two norms in the sequel that are not operator norms: the Frobenius norm, denoted by $\frobnorm{\mat{A}}$, and the column norm 
\[
\colnorm{\mat{A}}=\sum\nolimits_k\norm{\mat{a}_{k}}_{2}.
\]

The expectation of a random variable $X$ is written $\E X$ and its variance is written $\var{X} = \E(X-\E X)^2$. The expectation taken with respect to one variable $X$, with all others fixed, is written $\E_X$. The $L_q$ norm of $X$ is denoted by $\E^{q}(X)=(\E|X|^{q})^{1/q}.$ 

The expression $X \sim Y$ indicates the random variables $X$ and $Y$ are identically distributed. Given a random variable $X$, the symbol $X^\prime$ denotes a random variable independent of $X$ such that $X^\prime \sim X$. The indicator variable of the event $X > Y$ is written $\mathbb{1}_{X > Y}$.

The Bernoulli distribution with expectation $p$ is written $\text{Bern}(p)$ and the Binomial distribution of $n$ independent trials each with success probability $p$ is written $\text{Bin}(n,p)$. We write $X \sim \text{Bern}(p)$ to indicate $X$ is Bernoulli.

A Rademacher random variable takes on the values $\pm1$ with equal probability. A vector whose components are independent Rademacher variables is called a Rademacher vector. A real sum whose terms are weighted by independent Rademacher variables is called a Rademacher sum. The Khintchine inequality \cite{Szarek78} gives information on the moments of a Rademacher sum; in particular, it tells us the expected value of the sum is equivalent with the $\ell_2$ norm of the vector $\mat{x}$:

\begin{prop}[Khintchine inequality]
Let $\mat{x}$ be a real vector, and let $\boldsymbol{\varepsilon}$ be a Rademacher vector. Then 
\[
\frac{1}{\sqrt{2}}\norm{\mat{x}}_{2}\leq\E\left|\sum\nolimits_{k}\varepsilon_{k}x_{k}\right|\leq\norm{\mat{x}}_{2}.
\]
\end{prop}

\section{The $\infp$ norm of a Random Matrix}
\label{sec:inftopbound}

We are interested in schemes that approximate a given matrix $\mat{A}$ by means of a random matrix $\mat{X}$ in such a way that the entries of $\mat{X}$ are independent and $\E\mat{X} = \mat{A}$. It follows that the error matrix $\mat{Z}=\mat{A}-\mat{X}$ has independent, zero-mean entries. Our intellectual concern is the class of sparse random matrices, but this property does not play a role at this stage of the analysis.

In this section, we derive a bound on the expected value of the $\infp$ norm of a random matrix with independent, zero-mean entries. We also study the tails of this error. In the next two sections, we use the results of this section to reach more detailed conclusions on the $\infone$ and $\inftwo$ norms of $\mat{Z}$.

\subsection{Expected $\infp$ norm} The main tool used to derive the bound on the expected norm of $\mat{Z}$ is the following symmetrization argument \cite[Lemma 2.3.1 et seq.]{vdVW}.

\begin{prop}
Let $Z_1, \ldots, Z_n, Z_1^\prime, \ldots, Z_n^\prime$ be independent random variables satisfying $Z_i \sim Z_i^\prime$, and let $\boldsymbol{\varepsilon}$ be a Rademacher vector. Let $\mathcal{F}$ be a family of functions such that 
\[
\sup_{f \in \mathcal{F}} \sum\nolimits_{k=1}^n (f(Z_k) - f(Z_k^\prime))
\]
is measurable. Then 
\[
\E \sup_{f \in \mathcal{F}} \sum\nolimits_{k=1}^n (f(Z_k) - f(Z_k^\prime)) = \E \sup_{f \in \mathcal{F}} \sum\nolimits_{k=1}^n \varepsilon_k (f(Z_k) - f(Z_k^\prime)).
\]
\label{prop:symmetrization}
\end{prop}
Since we work with finite-dimensional probability models and linear functions, measurability concerns can be ignored.

The other critical tool is a version of Talagrand's Rademacher comparison theorem \cite[Theorem 4.12 et seq.]{LT91}.

\begin{prop}
Fix finite-dimensional vectors $\mat{z}_{1},\ldots,\mat{z}_{n}$ and let
$\boldsymbol{\varepsilon}$ be a Rademacher vector. Then \[
\E\supoverqball\sum\nolimits_{k=1}^{n}\varepsilon_k|\langle\mat{z}_{k},\mat{u}\rangle|\leq\E\supoverqball\sum\nolimits_{k=1}^{n}\varepsilon_{k}\langle\mat{z}_{k},\mat{u}\rangle.\]
 \label{prop:comparison}
\end{prop}

Now we state and prove the bound on the expected norm of $\mat{Z}$.
\begin{thm}
Let $\mat{Z}$ be a random matrix with independent, zero-mean entries and let $\boldsymbol{\varepsilon}$ be a Rademacher vector independent of $\mat{Z}$. Then
\[
\E\norm{\mat{Z}}_{\infty\rightarrow p} \leq 2\E\pnorm{\sum\nolimits_k \varepsilon_k \mat{z}_k} + 2 \supoverqball \E\sum\nolimits_k \left|\sum\nolimits_j \varepsilon_j Z_{jk} u_j \right|
\]
where $q$ is the conjugate exponent of $p$.
\label{thm:inftopbound}
\end{thm}

\begin{proof}[Proof of Theorem \ref{thm:inftopbound}] 
By duality, 
\[
\E\norm{\mat{Z}}_{\infty\rightarrow p}=\E\norm{\mat{Z}^{T}}_{q\rightarrow1}=\E\supoverqball\sum\nolimits _{k}|\langle\mat{z}_{k},\mat{u}\rangle|.
\]
Center the terms in the sum and apply subadditivity of the maximum to get
\begin{equation}
\begin{aligned}
\E\norm{\mat{Z}}_{\infty\rightarrow p} & \leq\E\supoverqball\sum\nolimits_{k}(|\langle\mat{z}_{k},\mat{u}\rangle|-\E^\prime|\langle\mat{z}_{k}^\prime,\mat{u}\rangle|)+\supoverqball\E\sum\nolimits_{k}|\langle\mat{z}_{k},\mat{u}\rangle| \\
	& \eqqcolon F+S.
\end{aligned}
\label{eqn:infpexpression}
\end{equation}

Begin with the first term in \eqref{eqn:infpexpression}. Use Jensen's inequality to draw the expectation outside of the maximum: 
\[
F\leq\E\supoverqball\sum\nolimits_{k}(|\langle\mat{z}_{k},\mat{u}\rangle|-|\langle\mat{z}_{k}^{\prime},\mat{u}\rangle|).
\]
Now apply Proposition \ref{prop:symmetrization} to symmetrize the random variable: 
\[
F\leq\E\supoverqball\sum\nolimits_{k}\varepsilon_{k}(|\langle\mat{z}_{k},\mat{u}\rangle|-|\langle\mat{z}_{k}^{\prime},\mat{u}\rangle|).
\]
By the subadditivity of the maximum, \[
F\leq\E\left(\supoverqball\sum\nolimits_{k}\varepsilon_{k}|\langle\mat{z}_{k},\mat{u}\rangle|+\supoverqball\sum\nolimits_{k}-\varepsilon_{k}|\langle\mat{z}_{k},\mat{u}\rangle|\right)=2\E\supoverqball\sum\nolimits_{k}\varepsilon_{k}|\langle\mat{z}_{k},\mat{u}\rangle|,
\]
where we have invoked the fact that $-\varepsilon_{k}$ has the Rademacher distribution. Apply Proposition \ref{prop:comparison} to get the final estimate of $F$: 
\[
F\leq2\E\supoverqball\sum\nolimits_{k}\varepsilon_{k}\langle\mat{z}_{k},\mat{u}\rangle=2\E\supoverqball\left\langle \sum\nolimits_{k}\varepsilon_{k}\mat{z}_{k},\mat{u}\right\rangle =2\E\norm{\sum\nolimits_{k}\varepsilon_{k}\mat{z}_{k}}_{p}.
\]

Now consider the last term in \eqref{eqn:infpexpression}. Use Jensen's inequality to prepare for symmetrization: 
\[
\begin{split}
S & =\supoverqball\E\sum\nolimits_{k}\left|\sum\nolimits_{j}Z_{jk}u_{j}\right|=\supoverqball\E\sum\nolimits_{k}\left|\sum\nolimits_{j}(Z_{jk}-\E^\prime Z_{jk}^\prime)u_{j}\right|\\
 & \leq\supoverqball\sum\nolimits_{k}\E\left|\sum\nolimits_{j}(Z_{jk}-Z_{jk}^{\prime})u_{j}\right|.
\end{split}
\]
Apply Proposition \ref{prop:symmetrization} to the expectation of the inner sum to see
\[
S\leq\supoverqball\sum\nolimits_{k}\E\left|\sum\nolimits_{j}\varepsilon_{j}(Z_{jk}-Z_{jk}^{\prime})u_{j}\right|.
\]
The triangle inequality gives us the final expression:
\[
S\leq\supoverqball2\E\sum\nolimits_{k}\left|\sum\nolimits_{j}\varepsilon_{j}Z_{jk}u_{j}\right|.
\]
Introduce the bounds for $F$ and $S$ into \eqref{eqn:infpexpression} to complete the proof.
\end{proof}

\subsection{Tail bound for $\infp$ norm}
We now develop a deviation bound for the $\infp$ approximation error. The argument is based on a bounded differences inequality.

First we establish some notation. Let $g:\R^{n}\rightarrow\R$ be a measurable function of $n$ random variables. Let $X_{1},\ldots,X_{n}$ be independent random variables, and write $W=g(X_{1},\ldots,X_{n})$. Let $W_i$ denote the random variable obtained by replacing the $i$th argument of $g$ with an independent copy: $W_{i}=g(X_{1},\dots,X_{i}^{\prime},\ldots,X_{n})$. 

The following bounded differences inequality states that if $g$ is insensitive to changes of a single argument, then $W$ does not deviate much from its mean.
\begin{prop}[\protect{\cite{BLM03}}]
Let $W$ and $\{W_i\}$ be random variables defined as above. Assume that there exists a positive number $C$ such that, almost surely, 
\[
\sum\nolimits_{i=1}^{n}(W-W_{i})^{2}\mathbb{1}_{W>W_{i}}\leq C.
\]
Then, for all $t>0,$
\[
\Prob{W>\E W+t}\leq \e^{-t^{2}/(4C)}.
\]
\label{prop:logsob}
\end{prop}

To apply Proposition \ref{prop:logsob}, we let $\mat{Z}=\mat{A}-\mat{X}$ be our error matrix, $W=\norm{\mat{Z}}_{\infty\rightarrow p}$, and $W^{jk}=\norm{\mat{Z}^{jk}}_{\infty\rightarrow p}$, where $\mat{Z}^{jk}$ is a matrix obtained by replacing $a_{jk}-X_{jk}$ with an identically
distributed variable $a_{jk}-X_{jk}^{\prime}$ while keeping all other variables fixed. The $\infp$ norms are sufficiently insensitive to each entry of the matrix that Proposition \ref{prop:logsob} gives us a useful deviation bound.

\begin{thm}
Fix an $m \times n$ matrix $\mat{A}$, and let $\mat{X}$ be a random matrix with independent entries for which $\E X = \mat{A}$. Assume $\left|X_{jk}\right|\leq\frac{D}{2}$ almost surely for all $j,k$. Then, for all $t>0$,
\[
\Prob{\norm{\mat{A}-\mat{X}}_{\infty\rightarrow p}> \E\norm{\mat{A}-\mat{X}}_{\infty\rightarrow p}+t}\leq \e^{-t^{2}/(4D^{2}nm^{s})}
\]
where $s=\max\{0, 1-2/q\}$ and $q$ is the conjugate exponent to $p$. 
\label{thm:inftopnormdevbnd}
\end{thm}

\begin{proof}
Let $q$ be the conjugate exponent of $p$, and choose $\mat{u},\mat{v}$ such that $W=\mat{u}^{T}\mat{Z}\mat{v}$ and $\norm{\mat{u}}_q = 1$ and $\norm{\mat{v}}_\infty = 1.$ Then 
\[
(W-W^{jk})\mathbb{1}_{W>W^{jk}}\leq\mat{u}^{T}\left(\mat{Z}-\mat{Z}^{jk}\right)\mat{v}\,\mathbb{1}_{W>W^{jk}}=(X_{jk}^{\prime}-X_{jk})u_{j}v_{k}\,\mathbb{1}_{W>W^{jk}}\leq D|u_{j}v_{k}|.
\]
This implies 
\[
\sum\nolimits_{j,k}(W-W^{jk})^{2}\mathbb{1}_{W>W^{jk}}\leq D^{2}\sum\nolimits_{j,k}|u_{j}v_{k}|^{2}\leq nD^{2}\norm{\mat{u}}_{2}^{2},
\]
so we can apply Proposition \ref{prop:logsob} if we have an estimate for $\norm{\mat{u}}_2^2$. We have the bounds $\norm{\mat{u}}_{2}\leq\norm{\mat{u}}_{q}$ for $q\in[1,2]$ and $\norm{\mat{u}}_{2}\leq m^{1/2-1/q}\norm{\mat{u}}_{q}$ for $q\in[2,\infty]$. Therefore,
\[
\sum\nolimits_{j,k}(W-W^{jk})^{2}\mathbb{1}_{W>W^{jk}}\leq D^{2}
\begin{cases}
nm^{1-2/q}, & q\in[2,\infty]\\
n, & q\in[1,2].
\end{cases}
\]
It follows from Proposition \ref{prop:logsob} that 
\[
\Prob{\norm{\mat{A}-\mat{X}}_{\infty\rightarrow p}>\E\norm{\mat{A} - \mat{X}}_{\infty\rightarrow p}+t}= \Prob{W>\E W+t}\leq \e^{-t^{2}/(4D^{2}nm^{s})}
\]
where $s=\max\left\{0,1-2/q\right\}.$
\end{proof}
It is often convenient to measure deviations on the scale of the mean. Taking $t=\delta \E\norm{\mat{A} - \mat{X}}_{\infty \rightarrow p}$ in Theorem $\ref{thm:inftopnormdevbnd}$ gives the following result.
\begin{cor}
Under the conditions of Theorem \ref{thm:inftopnormdevbnd}, for all $\delta>0$, 
\[
\Prob{\norm{\mat{A}-\mat{X}}_{\infty \rightarrow p}>(1 + \delta)\E\norm{\mat{A} - \mat{X}}_{\infty \rightarrow p}}\leq \e^{-\delta^{2}\left(\E\norm{\mat{A} - \mat{X}}_{\infty \rightarrow p}\right)^{2}/(4D^{2}nm^{s})}.
\]
\label{cor:inftopreldevbnd}
\end{cor}

\section{Approximation in the $\infone$ norm}

\label{sec:inf1norm}

In this section, we develop the $\infone$ error bound as a consequence of Theorem \ref{thm:inftopbound}. We then prove that one form of the error bound is optimal, and we describe an example of its application to matrix sparsification.

\subsection{Expected $\infone$ norm}
To derive the $\infone$ error bound, we first apply Theorem \ref{thm:inftopbound} with $p=1$.

\begin{thm}
Suppose that $\mat{Z}$ is a random matrix with independent, zero-mean entries. Then 
\[
\E\infonorm{\mat{Z}}\leq2\E(\colnorm{\mat{Z}}+\colnorm{\mat{Z}^{T}}).
\]
\label{thm:inf1errbound}
\end{thm}

\begin{proof}
Apply Theorem \ref{thm:inftopbound} to get 
\begin{equation}
\begin{aligned}
\E\norm{\mat{Z}}_{\infty\rightarrow1} & \leq2\E\norm{\sum\nolimits_{k}\varepsilon_{k}\mat{z}_{k}}_{1}+2\supoverinfball\E\sum\nolimits_{k}\left|\sum\nolimits_{j}\varepsilon_{j}Z_{jk}u_{j}\right| \\
 & \eqqcolon F + S.
\end{aligned}
\label{eqn:rawinf1bound}
\end{equation}
Use H\"older's inequality to bound the first term in \eqref{eqn:rawinf1bound} with a sum of squares:
\begin{align*}
F & =2\E\sum\nolimits_{j}\left|\sum\nolimits_{k}\varepsilon_{k}Z_{jk}\right|=2\E_{\mat{Z}}\sum\nolimits_{j}\E_{\boldsymbol{\varepsilon}}\left|\sum\nolimits_{k}\varepsilon_{k}Z_{jk}\right|\\
 & \leq2\E_{\mat{Z}}\sum\nolimits_{j}\left(\E_{\boldsymbol{\varepsilon}}\left|\sum\nolimits_{k}\varepsilon_{k}Z_{jk}\right|^{2}\right)^{1/2}.
\end{align*}
The inner expectation can be computed exactly by expanding the square and using the independence of the Rademacher variables:
\[
F \leq 2\E\sum\nolimits_{j}\left(\sum\nolimits_{k}Z_{jk}^{2}\right)^{1/2}=2\E\colnorm{\mat{Z}^{T}}.
\]

We treat the second term in the same manner. Use H\"older's inequality to replace the sum with a sum of squares and invoke the independence of the Rademacher variables to eliminate cross terms:
\[
S \leq2\supoverinfball\E_{\mat{Z}}\sum\nolimits_{k}\left(\E_{\boldsymbol{\varepsilon}}\left|\sum\nolimits_{j}\varepsilon_{j}Z_{jk}u_{j}\right|^{2}\right)^{1/2}
 = 2\supoverinfball\E\sum\nolimits_{k}\left(\sum\nolimits_{j}Z_{jk}^{2}u_{j}^{2}\right)^{1/2}.
\]
Since $\infnorm{\mat{u}}=1$, it follows that $u_{j}^{2}\leq1$ for all $j$, and 
\[
S \leq2\E\sum\nolimits_{k}\left(\sum\nolimits_{j}Z_{jk}^{2}\right)^{1/2}=2\E\colnorm{\mat{Z}}.
\]
Introduce these estimates for $F$ and $S$ into \eqref{eqn:rawinf1bound} to complete the proof.
\end{proof}

Taking $\mat{Z}=\mat{A}-\mat{X}$ in Theorem \ref{thm:inf1errbound},
we find 
\[
\E\infonorm{\mat{A}-\mat{X}}\leq2\E\left[\sum\nolimits_{k}\left(\sum\nolimits_{j}(a_{jk}-X_{jk})^{2}\right)^{1/2}+\sum\nolimits_{j}\left(\sum\nolimits_{k}(a_{jk}-X_{jk})^{2}\right)^{1/2}\right].\]

A simple application of Jensen's inequality gives an error bound in terms of the variances of the entries of $\mat{X}$. 
\begin{cor}
Fix the matrix $\mat{A}$, and let $\mat{X}$ be a random matrix with independent entries for which $\E X_{jk}=a_{jk}$. Then 
\[
\E\infonorm{\mat{A}-\mat{X}}\leq2\left[\sum\nolimits_{k}\left(\sum\nolimits_{j}\var{X_{jk}}\right)^{1/2}+\sum\nolimits_{j}\left(\sum\nolimits_{k}\var{X_{jk}}\right)^{1/2}\right].
\]
\label{cor:inf1errbound}
\end{cor}

\subsection{Optimality}

The bound in Corollary \ref{cor:inf1errbound} is optimal in the sense that there are families of matrices $\mat{A}$ and random approximants $\mat{X}$ for which $\E\infonorm{\mat{A}-\mat{X}}$ grows like one of the terms in the bound and dominates the other term in the bound. To show this, we construct specific examples.

Let $\mat{A}$ be a tall $m\times\sqrt{m}$ matrix of ones and choose the approximant $X_{jk} \sim 2\text{ Bern}\left(\tfrac{1}{2}\right).$
With this choice, $\var{X_{jk}}=1$, so the first term in the bound is $m$ and the second term is $m^{5/4}.$
The following argument from \cite[Sec. 4.2]{RV07} establishes that $\Vert\mat{A}-\mat{X}\Vert_{\infty\rightarrow1}$ grows like $m^{5/4}$.

Observe that the matrix $\mat{A}-\mat{X}=[\varepsilon_{jk}]$, where $\varepsilon_{jk}$ are i.i.d. Rademacher variables. Its $\infone$ norm is 
\[
\infonorm{\mat{A}-\mat{X}}=\max_{{{\infnorm{\mat{y}}=1\atop \infnorm{\mat{x}}=1}}}\sum\nolimits_{j,k}\varepsilon_{jk}x_{j}y_{k}.
\]
Let $\delta_{k}$ be a sequence of i.i.d. Rachemacher variables. By the scalar Khintchine inequality, 
\[
\E_{\boldsymbol{\delta}}\sum\nolimits_{j}\left|\sum\nolimits_{k}\varepsilon_{jk}\delta_{k}\right|\geq\sum\nolimits_{j}\frac{1}{\sqrt{2}}\|\boldsymbol{\varepsilon}_{j\cdot}\|_{2}=\frac{1}{\sqrt{2}}m^{5/4}.
\]
The probabilistic method shows that there is a sign vector $\mat{x}$ for which 
\[
\sum\nolimits_{j}\left|\sum\nolimits_{k}\varepsilon_{jk}x_{k}\right|\geq\frac{1}{\sqrt{2}}m^{5/4}.
\]
Choose the vector $\mat{y}$ with components $y_{j}=\sgn(\sum_{k}\varepsilon_{jk}x_{k}).$ Then 
\[
\norm{\mat{A}-\mat{X}}_{\infty\rightarrow1}\geq\sum\nolimits_{j}\sum\nolimits_{k}\varepsilon_{jk}y_{j}x_{k}=\sum\nolimits_{j}\left|\sum\nolimits_{k}\varepsilon_{jk}x_{k}\right|\geq\frac{1}{\sqrt{2}}m^{5/4}.
\]
This shows $\Vert\mat{A}-\mat{X}\Vert_{\infty\rightarrow1}$ grows like the second term in the error bound, so this term of the bound cannot be ignored.

These arguments, applied to a fat $\sqrt{n}\times n$ matrix of ones, also establish the necessity of the first term.

\subsection{Example application}

In this section we provide an example illustrating the application of Corollary \ref{cor:inf1errbound} to matrix sparsification.

From Corollary \ref{cor:inf1errbound} we infer that a good scheme for sparsifying a matrix $\mat{A}$ while minimizing the expected relative $\infone$ error is one which drastically increases the sparsity of $\mat{X}$ while keeping the relative error 
\[
\frac{\sum\nolimits_{k}\Big(\sum\nolimits_{j}\var{X_{jk}}\Big)^{1/2}+\sum\nolimits_{j}\Big(\sum\nolimits_{k}\var{X_{jk}}\Big)^{1/2}}{\infonorm{\mat{A}}}
\]
small. Once a sparsification scheme is chosen, the hardest part of estimating this quantity is probably estimating the $\infone$ norm of $\mat{A}$. The example shows, for a simple family of approximation schemes, what kind of sparsification results can be obtained using Corollary \ref{cor:inf1errbound} when we have a very good handle on this quantity. 

Consider the case where $\mat{A}$ is an $n\times n$ matrix whose entries all lie within an interval bounded away from zero; for definiteness, take them to be positive. Let $\gamma$ be a desired bound on the expected relative $\infone$ norm error. We choose the randomization strategy $X_{jk}\sim \frac{a_{jk}}{p}\text{ Bern}(p)$ and ask how small can $p$ be without violating our bound on the expected error.

In this case,
\[
\infonorm{\mat{A}}=\sum\nolimits_{j,k}a_{jk} = \asymO{n^{2}},
\]
and $\var{X_{jk}}=\frac{a_{jk}^{2}}{p}-a_{jk}^{2}.$ Consequently, the first term in Corollary \ref{cor:inf1errbound} satisfies
\[
\begin{aligned}
\sum\nolimits_{k}\left(\sum\nolimits_{j}\var{X_{jk}}\right)^{1/2} & =\sum\nolimits_{k}\left(\frac{1}{p}\norm{\mat{a}_{k}}_{2}^{2}-\norm{\mat{a}_{k}}_{2}^{2}\right)^{1/2}=\left(\frac{1-p}{p}\right)^{1/2}\colnorm{\mat{A}} \\
 & = \asymO{\left(\frac{1-p}{p}\right)^{1/2}n\sqrt{n}}
\end{aligned}
\]
and likewise the second term satisfies
\[
\sum\nolimits_{j}\left(\sum\nolimits_{k}\var{X_{jk}}\right)^{1/2} = \asymO{\left(\frac{1-p}{p}\right)^{1/2}n\sqrt{n}}.
\]
Therefore the relative $\infone$ norm error satisfies
\[
\frac{\sum\nolimits_{k}\Big(\sum\nolimits_{j}\var{X_{jk}}\Big)^{1/2}+\sum\nolimits_{j}\Big(\sum\nolimits_{k}\var{X_{jk}}\Big)^{1/2} }{\infonorm{\mat{A}}}= \asymO{\left(\frac{1-p}{pn}\right)^{1/2}}.
\]
It follows that $\E\infonorm{\mat{A}-\mat{X}}<\gamma$ for $p$ on the order of $(1+n\gamma^{2})^{-1}$ or larger. The expected number of nonzero entries in $\mat{X}$ is $pn^{2}$, so for matrices with this structure, we can sparsify with a relative $\infone$ norm error smaller than $\gamma$ while reducing the number of expected nonzero entries to as few as $\mathrm{O}(\frac{n^{2}}{1+n\gamma^{2}})=\mathrm{O}(\frac{n}{\gamma^{2}}).$ Intuitively, this sparsification result is optimal in the dimension: it seems we must keep on average at least one entry per row and column if we
are to faithfully approximate $\mat{A}$.

\section{Approximation in the $\inftwo$ norm}

\label{sec:inf2norm}

In this section, we develop the $\inftwo$ error bound stated in the introduction, establish the optimality of a related bound, and provide examples of its application to matrix sparsification. To derive the error bound, we first specialize Theorem \ref{thm:inftopbound} to the case of $p=2$.

\begin{thm}
Suppose that $\mat{Z}$ is a random matrix with independent, zero-mean entries. Then 
\[
\E\inftnorm{\mat{Z}}\leq2\E\frobnorm{\mat{Z}}+2\min_{\mat{D}}\E\norm{\mat{Z}\mat{D}^{-1}}_{2\rightarrow\infty}
\]
where $\mat{D}$ is a positive diagonal matrix that satisfies $\trace(\mat{D}^{2})=1$.
\label{thm:inf2errbound}
\end{thm}
\begin{proof}
Apply Theorem \ref{thm:inftopbound} to get 
\begin{equation}
\begin{aligned}
\E\norm{\mat{Z}}_{\infty\rightarrow2} & \leq2\E\norm{\sum\nolimits_{k}\varepsilon_{k}\mat{z}_{k}}_{2}+2\supovertball\E\sum\nolimits_{k}\left|\sum\nolimits_{j}\varepsilon_{j}Z_{jk}u_{j}\right| \\
	& \eqqcolon F + S.
\end{aligned}
\label{eqn:inf2boundraw}
\end{equation}
Expand the first term, and use Jensen's inequality to move the expectation with respect to the Rademacher variables inside the square root: 
\[
F = 2\E\left(\sum\nolimits_{j}\left|\sum\nolimits_{k}\varepsilon_{k}Z_{jk}\right|^{2}\right)^{1/2}\leq2\E_{\mat{Z}}\left(\sum\nolimits_{j}\E_{\boldsymbol{\varepsilon}}\left|\sum\nolimits_{k}\varepsilon_{k}Z_{jk}\right|^{2}\right)^{1/2}.
\]
The independence of the Rademacher variables implies that the cross terms cancel, so 
\[
F \leq2\E\left(\sum\nolimits_{j}\sum\nolimits_{k}Z_{jk}^{2}\right)^{1/2}=2\E\frobnorm{\mat{Z}}.
\]

We use the Cauchy--Schwarz inequality to replace the $\ell_{1}$ norm with an $\ell_{2}$ norm in the second term of \eqref{eqn:inf2boundraw}. A direct application would introduce a possibly suboptimal factor of $\sqrt{n}$ (where $n$ is the number of columns in $\mat{Z}$), so instead we choose $d_{k}>0$ such that $\sum_{k}d_{k}^{2}=1$ and use the corresponding weighted $\ell_{2}$ norm: 
\[
S = 2\supovertball\E\sum\nolimits_{k}\frac{\left|\sum\nolimits_{j}\varepsilon_{j}Z_{jk}u_{j}\right|}{d_{k}}d_{k}\leq2\supovertball\E\left(\sum\nolimits_{k}\frac{\left|\sum\nolimits_{j}\varepsilon_{j}Z_{jk}u_{j}\right|^{2}}{d_{k}^{2}}\right)^{1/2}.
\]
Move the expectation with respect to the Rademacher variables inside the square root and observe that the cross terms cancel:
\[
S\leq2\supovertball\E_{\mat{Z}}\left(\sum\nolimits_{k}\frac{\E_{\boldsymbol{\varepsilon}}\left|\sum\nolimits_{j}\varepsilon_{j}Z_{jk}u_{j}\right|^{2}}{d_{k}^{2}}\right)^{1/2}=2\supovertball\E\left(\sum\nolimits_{j,k}\frac{Z_{jk}^{2}u_{j}^{2}}{d_{k}^{2}}\right)^{1/2}.
\]
Use Jensen's inequality to pass the maximum through the expectation, and note that if $\norm{\mat{u}}_{2}=1$ then the vector formed by elementwise squaring $\mat{u}$ lies on the $\ell_{1}$ unit ball, thus 
\[
S \leq2\E\left(\max_{\norm{\mat{u}}_{1}=1}\sum\nolimits_{j,k}\left(\frac{Z_{jk}}{d_{k}}\right)^{2}u_{j}\right)^{1/2}.
\]
Clearly this maximum is achieved when $\mat{u}$ is chosen so $u_{j}=1$ at an index $j$ for which $\left(\sum\nolimits_{k}\left(\frac{Z_{jk}}{d_{k}}\right)^{2}\right)^{1/2}$ is maximal and $u_{j}=0$ otherwise. Consequently, the maximum is the largest of the $\ell_{2}$ norms of the rows of $\mat{Z}\mat{D}^{-1}$, where $\mat{D}=\text{diag}(d_{1},\ldots,d_{n})$. Recall that this quantity is, by definition, $\twoinfnorm{\mat{Z}\mat{D}^{-1}}.$ Therefore $S \leq2\E\twoinfnorm{\mat{Z}\mat{D}^{-1}}$. The theorem follows by optimizing our choice of $\mat{D}$ and introducing our estimates for $F$ and $S$ into \eqref{eqn:inf2boundraw}. 
\end{proof}

Taking $\mat{Z}=\mat{A}-\mat{X}$ in Theorem \ref{thm:inf2errbound}, we have  
\begin{equation}
\E\inftnorm{\mat{A}-\mat{X}}\leq2\E\left(\sum\nolimits_{j,k}(X_{jk}-a_{jk})^{2}\right)^{1/2}+2\min_{\mat{D}}\E\max_{j}\left(\sum\nolimits_{k}\frac{(X_{jk}-a_{jk})^{2}}{d_{k}^{2}}\right)^{1/2}.
\label{eqn:inf2errboundeq}
\end{equation}

We now derive a bound which depends only on the variances of the $X_{jk}$.
\begin{cor}
Fix the $m\times n$ matrix $\mat{A}$ and let $\mat{X}$ be a random matrix with independent entries so that $\E X=\mat{A}$. Then
\[
\E\inftnorm{\mat{A}-\mat{X}}\leq2\left(\sum\nolimits_{j,k}\var{X_{jk}}\right)^{1/2}+2\sqrt{m}\min_{\mat{D}}\max_{j}\left(\sum\nolimits_{k}\frac{\var{X_{jk}}}{d_{k}^{2}}\right)^{1/2}
\]
where $\mat{D}$ is a positive diagonal matrix with $\trace(\mat{D}^{2})=1$.
\label{cor:inf2errbound}
\end{cor}

\begin{proof}
Let $F$ and $S$ denote, respectively, the first and second term of \eqref{eqn:inf2errboundeq}. An application of Jensen's inequality shows that $F \leq2\left(\sum\nolimits_{j,k}\var{X_{jk}}\right)^{1/2}$.
A second application shows that 
\[
S \leq2\min_{\mat{D}}\left(\E\max_{j}\sum\nolimits_{k}\frac{(X_{jk}-a_{jk})^{2}}{d_{k}^{2}}\right)^{1/2}.
\]
Bound the maximum with a sum: 
\[
S \leq2\min_{\mat{D}}\left(\sum\nolimits_{j}\E\sum\nolimits_{k}\frac{(X_{jk}-a_{jk})^{2}}{d_{k}^{2}}\right)^{1/2}.
\]
The sum is controlled by a multiple of its largest term, so
\[
S \leq2\sqrt{m}\min_{\mat{D}}\left(\max_{j}\sum\nolimits_{k}\frac{\var{X_{jk}}}{d_{k}^{2}}\right)^{1/2},
\]
where $m$ is the number of rows of $\mat{A}.$
\end{proof}

\subsection{Optimality}

We now show that Theorem \ref{thm:inf2errbound} gives an optimal bound, in the sense that each of its terms is necessary. In the following, we reserve the letter $\mat{D}$ for a positive diagonal matrix with $\trace(\mat{D}^{2})=1.$

First, we establish the necessity of the Frobenius term by identifying a class of random matrices whose $\inftwo$ norms are larger than their weighted $\twoinf$ norms but comparable to their Frobenius norms. Let $\mat{Z}$ be a random $m\times\sqrt{m}$ matrix such that the entries in the first column of $\mat{Z}$ are equally likely to be positive or negative ones, and all other entries are zero. With this choice, $\E\inftnorm{\mat{Z}}=\E\frobnorm{\mat{Z}}=\sqrt{m}$. Meanwhile, $\E\twoinfnorm{\mat{Z}\mat{D}^{-1}}=\tfrac{1}{d_{11}}$, so $\min_{\mat{D}}\E\twoinfnorm{\mat{Z}\mat{D}^{-1}}=1$, which is much smaller than $\E\inftnorm{\mat{Z}}$. Clearly, the Frobenius term is necessary. 

Similarly, to establish the necessity of the weighted $\twoinf$ norm term, we consider a class of matrices whose $\inftwo$ norms are larger than their Frobenius norms but comparable to their weighted $\twoinf$ norms. Consider a $\sqrt{n}\times n$ matrix $\mat{Z}$ whose entries are all equally likely to be positive or negative ones. It is a simple task to confirm that $\E\inftnorm{\mat{Z}}\geq n$ and $\E\frobnorm{\mat{Z}}=n^{3/4}$; it follows that the weighted $\twoinf$ norm term is necessary. In fact, 
\[
\min_{\mat{D}}\E\twoinfnorm{\mat{Z}\mat{D}^{-1}}=\min_{\mat{D}}\E\max_{j=1,\ldots,\sqrt{n}}\left(\sum\nolimits_{k=1}^{n}\frac{Z_{jk}^{2}}{d_{kk}^{2}}\right)^{1/2}=\min_{\mat{D}}\left(\sum\nolimits_{k=1}^{n}\frac{1}{d_{kk}^{2}}\right)^{1/2}=n,
\]
so we see that $\E\inftnorm{\mat{Z}}$ and the weighted $\twoinf$ norm term are comparable.

\subsection{Example application}

From Theorem \ref{thm:inf2errbound} we infer that a good scheme for sparsifying a matrix $\mat{A}$ while minimizing the expected relative $\inftwo$ norm error is one which drastically increases the sparsity of $\mat{X}$ while keeping the relative error  
\[
\frac{\E\frobnorm{\mat{Z}}+\min_{\mat{D}}\E\twoinfnorm{\mat{Z}\mat{D}^{-1}}}{\inftnorm{\mat{A}}}
\]
small, where $\mat{Z} = \mat{A} - \mat{X}.$

As before, consider the case where $\mat{A}$ is an $n\times n$ matrix all of whose entries are positive and in an interval bounded away from zero. Let $\gamma$ be a desired bound on the expected relative $\inftwo$ norm error. We choose the randomization strategy $X_{jk}\sim\frac{a_{jk}}{p}\text{ Bern}(p)$ and ask how much can we sparsify while respecting our bound on the relative error. That is, how small can $p$ be? We appeal to Theorem \ref{thm:inf2errbound}. In this case, 
\[
\inftnorm{\mat{A}}=\left(\sum\nolimits_{j}\sum\nolimits_{k}a_{jk}^{2}+2\sum\nolimits_{j}\sum\nolimits_{\ell<m}a_{j\ell}a_{jm}\right)^{\frac{1}{2}}=\asymO{\left(n^{2}+n^{2}(n-1)\right)^{\frac{1}{2}}}.
\]
By Jensen's inequality, 
\[
\E\frobnorm{\mat{Z}} \leq \E\frobnorm{\mat{A}} + \E\frobnorm{\mat{X}}\leq \left(1 + \frac{1}{\sqrt{p}}\right)\frobnorm{\mat{A}} = \asymO{n \left(1 + \frac{1}{\sqrt{p}}\right)}.
\]
We bound the other term in the numerator, also using Jensen's inequality: 
\[
\min_{\mat{D}} \E\twoinfnorm{\mat{Z}\mat{D}^{-1}} \leq \sqrt{n}\E\twoinfnorm{\mat{Z}} \leq \sqrt{n}\left(1 + \frac{1}{\sqrt{p}}\right)\twoinfnorm{\mat{A}} = \asymO{n \left(1 + \frac{1}{\sqrt{p}}\right) } 
\]
to get 
\[
\frac{\E\frobnorm{\mat{Z}} + \min_{\mat{D}} \E\twoinfnorm{\mat{Z}\mat{D}^{-1}} }{\inftnorm{\mat{A}}} = \asymO{\frac{1}{\sqrt{n}} + \frac{1}{\sqrt{pn}}} = \asymO{\frac{1}{\sqrt{pn}}}
\]
We conclude that, for this class of matrices and this family of sparsification schemes, we can reduce the number of expected nonzero terms to $\asymO{\frac{n}{\gamma^{2}}}$ while maintaining an expected $\inftwo$ norm relative error of $\gamma$.

\section{Spectral error bound}

\label{sec:spectralerrbound}

In this section we establish a bound on $\E\norm{\mat{A}-\mat{X}}$ as an immediate consequence of Lata\l{}a's result \cite{Lat04}. We then derive a deviation inequality for the spectral approximation error using a log-Sobolev inequality from \cite{BLM03}, and use it to compare our results to those of Achlioptas and McSherry \cite{AM07} and Arora, Hazan, and Kale \cite{AHK06}.

\begin{thm}
Suppose $\mat{A}$ is a fixed matrix, and let $\mat{X}$ be a random matrix with independent entries for which $\E X=\mat{A}$. Then 
\[
\E\norm{\mat{A}-\mat{X}}\leq\mathrm{C}\left[\max_{j}\left(\sum\nolimits_{k}\var{X_{jk}}\right)^{1/2}+\max_{k}\left(\sum\nolimits_{j}\var{X_{jk}}\right)^{1/2}+\left(\sum\nolimits_{jk}\E(X_{jk}-a_{jk})^{4}\right)^{1/4}\right]
\]
where $\mathrm{C}$ is a universal constant.
\label{thm:spectralbound}
\end{thm}
In \cite{Lat04}, Lata\l{}a considered the spectral norm of random matrices with independent, zero-mean entries, and he showed that, for any such matrix $\mat{Z}$, 
\[
\E\Vert\mat{Z}\Vert\leq\mathrm{C}\left[\max_{j}\left(\sum\nolimits_{k}\E Z_{jk}^{2}\right)^{1/2}+\max_{k}\left(\sum\nolimits_{j}\E Z_{jk}^{2}\right)^{1/2}+\left(\sum\nolimits_{jk}\E Z_{jk}^{4}\right)^{1/4}\right],
\]
where $\mathrm{C}$ is some universal constant. Unfortunately, no estimate for $\mathrm{C}$ is available. Theorem \ref{thm:spectralbound} follows from Lata\l{}a's result, by taking $\mat{Z}=\mat{A}-\mat{X}$.

The bounded differences argument from Section \ref{sec:inftopbound} establishes the correct (subgaussian) tail behavior of $\mathbb{E}\norm{\mat{A}-\mat{X}}$.

\begin{thm}
Fix the matrix $\mat{A}$, and let $\mat{X}$ be a random matrix with independent entries for which $\E X = \mat{A}$. Assume $\left|X_{jk}\right|\leq D/2$ almost surely for all $j,k$. Then, for all $t>0$,
\[
\Prob{\norm{\mat{A}-\mat{X}}>\E\norm{\mat{A} - \mat{X}}+t}\leq \e^{-t^{2}/(4D^{2})}.
\]
\label{thm:normconcentration}
\end{thm}
\begin{proof}
The proof is exactly that of Theorem \ref{thm:inftopnormdevbnd}, except now $\mat{u}$ and $\mat{v}$ are both in the $\ell_2$ unit sphere.
\end{proof}

We find it convenient to measure deviations on the scale of the mean.
\begin{cor}
Under the conditions of Theorem \ref{thm:normconcentration}, for all $\delta>0$, 
\[
\Prob{\norm{\mat{A}-\mat{X}}>(1+\delta)\mathbb{E}\norm{\mat{A}-\mat{X}}}\leq \e^{-\delta^{2}\left(\mathbb{E}\norm{\mat{A}-\mat{X}}\right)^{2}/(4D^{2})}.
\]
\label{cor:devbnd}
\end{cor}

\subsection{Comparison with previous results}

To demonstrate the applicability of our bound on the spectral norm error, we consider the sparsification and quantization schemes used by Achlioptas and McSherry \cite{AM07}, and the quantization scheme proposed by Arora, Hazan, and Kale \cite{AHK06}. We show that our spectral norm error bound and the associated concentration result give results of the same order, with less effort. Throughout these comparisons, we take $\mat{A}$ to be a $m\times n$ matrix, with $m<n$, and we define $b=\max_{jk}|a_{jk}|$.

\subsubsection{A matrix quantization scheme}

First we consider the scheme proposed by Achlioptas and McSherry for
quantization of the matrix entries: 
\[
X_{jk}=\begin{cases}
b & \text{ with probability }\frac{1}{2}+\frac{a_{jk}}{2b}\\
-b & \text{ with probability }\frac{1}{2}-\frac{a_{jk}}{2b}\end{cases}.
\]
With this choice $\var{X_{jk}}=b^{2}-a_{jk}^{2}\leq b^{2}$, and $\E(X_{jk}-a_{jk})^{4}=b^{2}-3a^{4}+2a^{2}b^{2}\leq3b^{4}$, so the expected spectral error satisfies 
\[
\E\Vert\mat{A}-\mat{X}\Vert\leq \mathrm{C}(\sqrt{n}b+\sqrt{m}b+b\sqrt[4]{3mn})\leq4\mathrm{C}b\sqrt{n}.
\]
Applying Corollary \ref{cor:devbnd}, we find that the error satisfies 
\[
\Prob{\norm{\mat{A}-\mat{X}}>4\mathrm{C}b\sqrt{n}(1 + \delta)}\leq \e^{-\delta^{2}\mathrm{C}^{2}n}.
\]
In particular, with probability at least $1-\exp(-\mathrm{C}^{2}n)$,
\[
\norm{\mat{A}-\mat{X}}\leq8\mathrm{C}b\sqrt{n}.
\]
Achlioptas and McSherry proved that for $n\geq n_{0}$, where $n_{0}$ is on the order of $10^{9}$, with probability at least $1-\exp(-19(\log n)^{4})$,
\[
\norm{\mat{A}-\mat{X}}<4b\sqrt{n}.
\]
Thus, Theorem \ref{thm:normconcentration} provides a bound of the same order in $n$ which holds with higher probability and over a larger range of $n$.

\subsubsection{A nonuniform sparsification scheme}
Next we consider an analog to the nonuniform sparsification scheme proposed in the same paper. Fix a number $p$ in the range $(0,1)$ and sparsify entries with probabilities proportional to their magnitudes: 
\[
X_{jk}\sim \frac{a_{jk}}{p_{jk}}\text{ Bern}(p_{jk}), \text{ where } p_{jk}=\max\left\{ p\left(\frac{a_{jk}}{b}\right)^{2},\sqrt{p\left(\frac{a_{jk}}{b}\right)^{2}\times(8\log n)^{4}/n}\right\}.
\]
Achlioptas and McSherry determine that, with probability at least $1 - \exp(-19(\log n)^4)$,
\[
\norm{\mat{A}-\mat{X}}<4b\sqrt{n/p}.
\]
Further, the expected number of nonzero entries in $\mat{X}$ is less than
\begin{equation}
pmn\times\text{Avg}[(a_{jk}/b)^{2}]+m(8\log n)^{4}.
\end{equation}

Their choice of $p_{jk}$, in particular the insertion of the $(8\log n)^{4}/n$ factor, is an artifact of their method of proof. Instead, we consider a scheme which compares the magnitudes of $a_{jk}$ and $b$ to determine $p_{jk}$. Introduce the quantity $R=\max_{a_{jk}\neq0}b/|a_{jk}|$ to
measure the spread of the entries in $\mat{A}$, and take 
\[
X_{jk} \sim 
\begin{cases}
\frac{a_{jk}}{p_{jk}}\text{ Bern}(p_{jk}),\text{ where }p_{jk}=\frac{pa_{jk}^{2}}{pa_{jk}^{2}+b^{2}}, & a_{jk}\neq0\\
0, & a_{jk}=0.
\end{cases}
\]
With this scheme, $\var{X_{jk}}=0$ when $a_{jk}=0$, otherwise $\var{X_{jk}}=b^{2}/p$.
Likewise, $\E(X_{jk}-a_{jk})^{4}=0$ if $a_{jk}=0$, otherwise 
\[
\E(X_{jk}-a_{jk})^{4}\leq\var{X_{jk}}\norm{X_{jk}-a_{jk}}_{\infty}^{2}=\frac{b^{2}}{p}\max\left\{|a_{jk}|,|a_{jk}|\left(\frac{pa_{jk}^{2}+b^{2}}{pa_{jk}^{2}}-1\right)\right\}^{2}\leq\frac{b^{4}}{p^{2}}R^{2},
\]
so \[
\E\norm{\mat{A}-\mat{X}}\leq \mathrm{C}\left(b\sqrt{\frac{n}{p}}+b\sqrt{\frac{m}{p}}+b\sqrt{\frac{R}{p}}\sqrt[4]{mn}\right)\leq \mathrm{C}(2+\sqrt{R})b\sqrt{\frac{n}{p}}.
\]
Applying Corollary \ref{cor:devbnd}, we find that 
with probability at least $1-\exp(-\mathrm{C}^{2}(2+\sqrt{R})^{2}pn/16)$,
\[
\norm{\mat{A}-\mat{X}}\leq2\mathrm{C}(2+\sqrt{R})b\sqrt{\frac{n}{p}}.
\]
Thus, Theorem \ref{thm:spectralbound} and Achlioptas and McSherry's scheme-specific analysis yield results of the same order in $n$ and $p$. As before, we see that our bound holds with higher probability and over a larger range of $n$. Furthermore, since the expected number of nonzero entries in $\mat{X}$ satisfies 
\[
\sum\nolimits_{jk}p_{jk} = \sum\nolimits_{jk} \frac{pa_{jk}^2}{pa_{jk}^2 + b^2}\leq pnm\times \text{Avg}\left[\left(\frac{a_{jk}}{b}\right)^2\right],
\]
we have established a smaller limit on the expected number of nonzero entries.  

\subsubsection{A scheme which simultaneously sparsifies and quantizes}

Finally, we use Theorem \ref{thm:normconcentration} to estimate the error in using the scheme from \cite{AHK06} which simultaneously quantizes and sparsifies. Fix $\delta>0$ and consider 
\[
X_{jk}=\begin{cases}
\sgn(a_{jk})\frac{\delta}{\sqrt{n}}\text{ Bern}\left(\frac{|a_{jk}|\sqrt{n}}{\delta}\right), & |a_{jk}|\leq\frac{\delta}{\sqrt{n}}\\
a_{jk}, & \text{ otherwise}.\end{cases}
\]
Then $\var{X_{jk}}=0$ if $|a_{jk}|\geq\delta/\sqrt{n}$,
otherwise 
\[
\var{X_{jk}}=|a_{jk}|^{3}\frac{\sqrt{n}}{\delta}-2a_{jk}^{2}+|a_{jk}|\frac{\delta}{\sqrt{n}}\leq\frac{\delta^{2}}{n}.
\]
Also the fourth moment term is zero for large enough $a_{jk}$, otherwise
\[
\E(X_{jk}-a_{jk})^{4}=|a_{jk}|^{5}\frac{\sqrt{n}}{\delta}-4a_{jk}^{4}+6|a_{jk}|^{3}\frac{\delta}{\sqrt{n}}-4a_{jk}^{2}\frac{\delta^{2}}{n}+|a_{jk}|\left(\frac{\delta}{\sqrt{n}}\right)^{3}\leq8\frac{\delta^{4}}{n^{2}}.
\]
This gives the estimates 
\[
\E\norm{\mat{A}-\mat{X}}\leq C\left(\sqrt{n}\frac{\delta}{\sqrt{n}}+\sqrt{m}\frac{\delta}{\sqrt{n}}+2\frac{\delta}{\sqrt{n}}\sqrt[4]{mn}\right)\leq4C\delta
\]
and 
\[
\Prob{\norm{\mat{A}-\mat{X}}>4C\delta(\gamma+1)}\leq \e^{-\gamma^{2}C^{2}n}.
\]
Taking $\gamma=1$, we see that with probability at least $1-\exp(-C^{2}n),$
\[
\norm{\mat{A}-\mat{X}}\leq8C\delta.
\]
Let $S=\sum\nolimits_{j,k}|A_{jk}|$, then appealing to Lemma 1 in \cite{AHK06}, we find that $\mat{X}$ has $\asymO{\tfrac{\sqrt{n} S}{\gamma}}$ nonzero entries with probability at least $1-\exp\left(-\Omega\left(\tfrac{\sqrt{n}S}{\gamma}\right)\right)$.

Arora, Hazan, and Kale establish that this scheme guarantees $\norm{\mat{A}-\mat{X}}\leq\asymO{\delta}$ with probability at least $1-\exp(-\Omega(n))$, so we see that our general bound recovers a bound of the same order.

In conclusion, we see that the bound on expected spectral error in Theorem \ref{thm:normconcentration} in conjunction with the deviation result in Corollary \ref{cor:devbnd} provide guarantees comparable to those derived with scheme-specific analyses. We anticipate that the flexibility demonstrated here will make these useful tools for analyzing and guiding the design of novel sparsification schemes.

\nocite{Seg00} \bibliographystyle{amsalpha}
\bibliography{matrix_sparsification_errorbounds}

\end{document}